\documentclass[12pt,fleqn,draft]{article} 
\usepackage{amsmath,amssymb,amsthm,amsfonts,bm}
\usepackage{enumerate}
\usepackage{indentfirst}
\usepackage{cite}
\usepackage[usenames,dvipsnames]{color}

\makeatletter
\def\@cite#1#2{[{{\bfseries #1}\if@tempswa , #2\fi}]}
\renewcommand{\section}{%
\@startsection{section}{1}{\z@}
{0.5truecm plus -1ex minus -.2ex}%
{1.0ex plus .2ex}{\bfseries\large}}
\def\@seccntformat#1{\csname the#1\endcsname.\ }
\makeatother

\numberwithin{equation}{section} 
\theoremstyle{theorem}
\newtheorem{thm}{Theorem}[section]

\newtheorem{lem}[thm]{Lemma}

\theoremstyle{definition}
\newtheorem{df}{Definition}[section]
\newtheorem{remark}{Remark}[section]

\newtheorem*{prth1.1}{Proof of Theorem 1.1}
\newtheorem*{prth1.2}{Proof of Theorem 1.2}
\newtheorem*{prth1.3}{Proof of Theorem 1.3}
\newtheorem*{prth1.4}{Proof of Theorem 1.4}

\newcommand{\ep}{\varepsilon}

\let\widehat\widehat

\def\Pi{\widehat\pi}

\begin{document}
\begin{center}
    \Large{{\bf Existence for a singular nonlocal phase field system \\ 
            with inertial term}}
\end{center}
\vspace{5pt}
\begin{center}
    Shunsuke Kurima%
    \\
    \vspace{2pt}
    Department of Mathematics, 
    Tokyo University of Science\\
    1-3, Kagurazaka, Shinjuku-ku, Tokyo 162-8601, Japan\\
    {\tt shunsuke.kurima@gmail.com}\\
    \vspace{2pt}
\end{center}
\begin{center}    
    \small \today
\end{center}

\vspace{2pt}
\newenvironment{summary}
{\vspace{.5\baselineskip}\begin{list}{}{%
     \setlength{\baselineskip}{0.85\baselineskip}
     \setlength{\topsep}{0pt}
     \setlength{\leftmargin}{12mm}
     \setlength{\rightmargin}{12mm}
     \setlength{\listparindent}{0mm}
     \setlength{\itemindent}{\listparindent}
     \setlength{\parsep}{0pt}
     \item\relax}}{\end{list}\vspace{.5\baselineskip}}
\begin{summary}
{\footnotesize {\bf Abstract.} 
In this paper we deal with a singular nonlocal phase field system with inertial term.  
The system has the logarithm of the absolute temperature $\theta$ 
under time derivative.   
Although the system has a difficult mathematical point  
caused by the combination of $(\ln \theta)_{t}$, 
the inertial term and the nonlocal diffusion term for the order parameter $\varphi$      
(see Section \ref{Sec1.1}), 
we can establish existence of solutions by a key estimate 
(see Remark \ref{Keypoints}).  
}
\end{summary}
\vspace{10pt}

\newpage

\section{Introduction}\label{Sec1}

\subsection{Previous works}\label{Sec1.1}

The phase field system 
\begin{equation}\label{E}\tag{E}
\begin{cases}
(\alpha(\theta))_{t} + \ell\varphi_{t} - \eta\Delta \theta = f  
&\mbox{in}\ \Omega \times (0, T), 
\\[2.5mm] 
\zeta\varphi_{tt} + \varphi_t + A\varphi 
                                           + \beta(\varphi) + \pi(\varphi) = \ell\theta
&\mbox{in}\ \Omega \times (0, T) 
\end{cases}
\end{equation}
has been studied, 
where $\Omega \subset \mathbb{R}^d$ ($d= 1, 2, 3$) is a bounded domain, 
$T>0$, $\ell, \eta >0$, $\zeta \in \{0, 1\}$, 
$\alpha : D(\alpha) \subset \mathbb{R} \to \mathbb{R}$ 
is a single-valued maximal monotone function, 
$A : D(A) \subset L^2(\Omega) \to L^2(\Omega)$ is an operator, 
$\beta : \mathbb{R} \to \mathbb{R}$ is a single-valued maximal monotone function, 
$\pi : \mathbb{R} \to \mathbb{R}$ is an anti-monotone function, 
and $f : \Omega\times(0, T) \to \mathbb{R}$ is a given function.  
In particular, the following four case were studied:  
\begin{enumerate}
\setlength{\itemsep}{2mm}
\item $\zeta = 0$, $\alpha(\theta) = \theta$, $A\varphi = -\Delta\varphi$  
     \   (see e.g., \cite{EllZheng, CC2012, M2015, MN2018, CK1}).    
\item $\zeta = 0$, 
$\alpha(\theta) = \ln \theta$, $A\varphi = -\Delta \varphi$ 
\ (see e.g., \cite{C2018, CC2018}).      
\item $\zeta = 1$, $\alpha(\theta) = \theta$, $A\varphi = -\Delta \varphi$ 
\ (see e.g., \cite{WGZ2007, WGZ2007dynamicalBD, K4}).    
\item $\zeta = 1$, 
$\alpha(\theta) = \theta$, $A\varphi = a(\cdot)\varphi - J\ast\varphi$ 
\ (see e.g., \cite{GPS, K7}).   
\end{enumerate}
Here $J : \mathbb{R}^d \to \mathbb{R}$ is an interaction kernel, 
$a(x) := \int_{\Omega} J(x-y)\,dy$ and  
$(J\ast\varphi)(x) := \int_{\Omega} J(x-y)\varphi(y)\,dy$ 
for $x \in \Omega$.  
However, in the case that $\zeta = 1$, 
$\alpha(\theta) = \ln \theta$, $A\varphi = a(\cdot)\varphi - J\ast\varphi$,  
the system \eqref{E} seems to be not studied yet. 
In Cases 3 and 4 
to establish estimates for $\beta(\varphi)$ is more difficult 
compared to Cases 1 and 2 by the inertial term $\varphi_{tt}$.   
In Case 3, 
assuming that $|\beta''(r)| \leq C_{\beta}(1 + |r|)$ for all $r \in \mathbb{R}$, 
where $C_{\beta} > 0$ is some constant, we can obtain an estimate for $\beta(\varphi)$ 
by deriving the $L^{\infty}(0, T; H^1(\Omega))$-estimate for $\varphi$ 
and by the continuity of the embedding $H^1(\Omega) \hookrightarrow L^6(\Omega)$. 
On the other hand, 
in Case 4, 
since the regularity of $\varphi$ is lower 
compared to the case that $A\varphi = -\Delta \varphi$, 
it seems to be difficult to obtain estimates for $\beta(\varphi)$ 
in the same way as in Case 3. 
In Case 4, 
assuming that $\varphi_{0}, v_{0} \in L^{\infty}(\Omega)$,  
we can derive   
the $L^{\infty}(\Omega\times(0, T))$-estimate for $\varphi$ 
by establishing the $L^2(0, T; H^2(\Omega))$-estimate for $\theta$ 
and by the continuity of 
the embedding $H^2(\Omega) \hookrightarrow L^{\infty}(\Omega)$, 
and hence we can obtain an estimate for $\beta(\varphi)$. 
However, in the case that $\zeta = 1$, 
$\alpha(\theta) = \ln \theta$, $A\varphi = a(\cdot)\varphi - J\ast\varphi$, 
since the regularity of $\theta$ is lower  
compared to the case that $\alpha(\theta) = \theta$, 
it seems to be difficult to derive the $L^2(0, T; H^2(\Omega))$-estimate for $\theta$ 
in the same way as in Case 4. 
In this paper we try to solve the mathematical problem  
caused by the combination of $(\ln \theta)_{t}$, $\varphi_{tt}$ 
and $a(\cdot)\varphi - J\ast\varphi$ (see Remark \ref{Keypoints}).

\subsection{Main problem}

In this paper we consider the singular nonlocal phase field system with inertial term
\begin{equation}\label{P}\tag{P}
\begin{cases}
(\ln \theta)_{t} + \ell\varphi_{t} - \eta\Delta \theta = f  
&\mbox{in}\ \Omega \times (0, T), 
\\[1mm] 
\varphi_{tt} + \varphi_t + a(\cdot)\varphi - J\ast\varphi 
                                           + \beta(\varphi) + \pi(\varphi) = \ell\theta
&\mbox{in}\ \Omega \times (0, T), 
\\[2mm]
\partial_\nu \theta = 0 
&\mbox{on}\ \partial\Omega \times (0, T), 
\\[1mm] 
(\ln \theta)(0) = \ln \theta_0,\ \varphi(0) = \varphi_0,\ \varphi_{t}(0) = v_0 
&\mbox{in}\ \Omega,     
\end{cases}
\end{equation}
where $\Omega \subset \mathbb{R}^d$ ($d= 1, 2, 3$) is a bounded domain 
with smooth boundary $\partial\Omega$, 
$\partial_\nu$ denotes differentiation with respect to
the outward normal of $\partial\Omega$, 
$\theta_{0} : \Omega \to \mathbb{R}$, 
$\varphi_{0} : \Omega \to \mathbb{R}$ and 
$v_{0} : \Omega \to \mathbb{R}$   
are given functions. 
Moreover,  
we assume the four conditions:  
\begin{enumerate} 
\setlength{\itemsep}{0mm}
\item[(C1)] 
$J(-x) = J(x)$ for all $x \in \mathbb{R}^d$ 
and $\displaystyle\sup_{x \in \Omega} \int_{\Omega} |J(x-y)|\,dy < + \infty$. 
\item[(C2)] $\beta : \mathbb{R} \to \mathbb{R}$                                
is a single-valued maximal monotone function 
such that 
there exists a proper lower semicontinuous convex function 
$\widehat{\beta} : \mathbb{R} \to [0, +\infty)$ 
satisfying that   
$\widehat{\beta}(0) = 0$ and 
$\beta = \partial\widehat{\beta}$, 
where $\partial\widehat{\beta}$  
is the subdifferential of $\widehat{\beta}$. 
Moreover, $\beta : \mathbb{R} \to \mathbb{R}$ is local Lipschitz continuous. 
\item[(C3)] $\pi : \mathbb{R} \to \mathbb{R}$ is a Lipschitz continuous function. 
\item[(C4)] $f \in L^2(\Omega\times(0, T)) \cap L^1(0, T; L^{\infty}(\Omega))$,   
$\theta_0 \in L^{2}(\Omega)$, 
$\ln \theta_{0} \in L^2(\Omega)$,    
$\varphi_0, v_0 \in L^{\infty}(\Omega)$. 
\end{enumerate}

\medskip

%
%
%
Let us define the Hilbert spaces 
   \[
   H:=L^2(\Omega), \quad V:=H^1(\Omega)
   \]
 with inner products 
 \begin{align*} 
 &(u_{1}, u_{2})_{H}:=\int_{\Omega}u_{1}u_{2}\,dx \quad  (u_{1}, u_{2} \in H), \\
 &(v_{1}, v_{2})_{V}:=
 \int_{\Omega}\nabla v_{1}\cdot\nabla v_{2}\,dx + \int_{\Omega} v_{1}v_{2}\,dx \quad 
 (v_{1}, v_{2} \in V),
\end{align*}
 respectively,
 and with the related Hilbertian norms. 
 Moreover, we use the notation
   \[
   W:=\bigl\{z\in H^2(\Omega)\ |\ \partial_{\nu}z = 0 \quad 
   \mbox{a.e.\ on}\ \partial\Omega\bigr\}.
   \] 
The notation $V^{*}$ denotes the dual space of $V$ with 
 duality pairing $\langle\cdot, \cdot\rangle_{V^*, V}$. 

\bigskip

We define weak solutions of \eqref{P} as follows.
%
%
%
 \begin{df}         
 A pair $(\theta, \varphi)$ with 
    \begin{align*}
    &\theta \in L^2(0, T; V),\ \ln \theta \in H^1(0, T; V^{*}) \cap L^{\infty}(0, T; H), \\ 
    &\varphi \in W^{2, 2}(0, T; H) \cap W^{1, \infty}(0, T; L^{\infty}(\Omega)) 
    \end{align*}
 is called a {\it weak solution} of \eqref{P} 
 if $(\theta, \varphi)$ satisfies 
    \begin{align*}
        &\langle (\ln \theta)_{t}, w \rangle_{V^{*}, V}
           + \ell(\varphi_t, w)_{H} + \eta\int_{\Omega} \nabla\theta\cdot\nabla w = (f, w)_{H}  
       \notag \\ 
                &\hspace{75mm}   \mbox{a.e.\ on}\ (0, T) 
                           \ \  \mbox{for all}\ w \in V,  
     \\[2mm]
        &\varphi_{tt} + \varphi_t +a(\cdot)\varphi - J\ast\varphi 
                                                             + \beta(\varphi) + \pi(\varphi) = \ell\theta 
                               \quad \mbox{a.e.\ on}\ \Omega\times(0, T), 
     \\[2mm]
        &(\ln \theta)(0) = \ln \theta_0,\ \varphi(0) = \varphi_0,\ \varphi_{t}(0) = v_0 
                                                       \quad \mbox{a.e.\ on}\ \Omega. 
     \end{align*}
 \end{df}

\smallskip

\noindent 
The following theorem is concerned with existence of weak solutions to \eqref{P}.  
\begin{thm}\label{maintheorem1}
Assume that {\rm (C1)-(C4)} hold. 
Then there exists a weak solution $(\theta, \varphi)$ of \eqref{P}. 
\end{thm}

\subsection{Approximations}

In reference to \cite{CC2018}, to prove existence for \eqref{P} 
we consider the approximation 
\begin{equation*}\tag*{(P)$_{\ep}$}\label{Pep}
     \begin{cases}
         (\ep\theta_{\ep} + \ln \theta_{\ep})_{t} + \ell(\varphi_{\ep})_{t}  
         - \eta\Delta\theta_{\ep} = f   
         & \mbox{in}\ \Omega\times(0, T), 
 \\[2mm]
         (\varphi_{\ep})_{tt} + (\varphi_{\ep})_{t} 
         + a(\cdot)\varphi_{\ep} - J\ast\varphi_{\ep} 
                                   + \beta(\varphi_{\ep}) + \pi(\varphi_{\ep}) = \ell\theta_{\ep}
         & \mbox{in}\ \Omega\times(0, T), 
 \\[2mm]
         \partial_{\nu}\theta_{\ep} = 0                                   
         & \mbox{on}\ \partial\Omega\times(0, T),
 \\[2mm]
        (\ep\theta_{\ep} + \ln \theta_{\ep})(0) 
        = \ep\theta_{0} + \ln \theta_{0},\ 
        (\varphi_{\ep})(0) = \varphi_0,\ (\varphi_{\ep})_{t}(0) = v_0                                    
         & \mbox{in}\ \Omega, 
     \end{cases}
 \end{equation*}
where $\ep \in (0, 1]$. 
The definition of weak solutions to \ref{Pep} is as follows.  
\begin{df}         
 A pair $(\theta_{\ep}, \varphi_{\ep})$ with 
    \begin{align*}
    &\theta_{\ep} \in L^2(0, T; V) \cap L^{\infty}(0, T; H),\   
      \ep\theta_{\ep} + \ln \theta_{\ep} \in H^1(0, T; V^{*}),\  
      \ln \theta_{\ep} \in L^{\infty}(0, T; H),      \\       
    &\varphi_{\ep} \in W^{2, 2}(0, T; H) \cap W^{1, \infty}(0, T; L^{\infty}(\Omega)) 
    \end{align*}
 is called a {\it weak solution} of \ref{Pep} 
 if $(\theta_{\ep}, \varphi_{\ep})$ satisfies 
    \begin{align*}
        &\langle (\ep\theta_{\ep} + \ln \theta_{\ep})_{t}, w \rangle_{V^{*}, V}
           + \ell((\varphi_{\ep})_t, w)_{H} 
           + \eta\int_{\Omega} \nabla\theta_{\ep}\cdot\nabla w = (f, w)_{H}  
       \notag \\ 
                &\hspace{75mm}   \mbox{a.e.\ on}\ (0, T) 
                           \ \  \mbox{for all}\ w \in V,  
     \\[2mm]
        &(\varphi_{\ep})_{tt} + (\varphi_{\ep})_t +a(\cdot)\varphi_{\ep} - J\ast\varphi_{\ep} 
                                   + \beta(\varphi_{\ep}) + \pi(\varphi_{\ep}) = \ell\theta_{\ep} 
                               \quad \mbox{a.e.\ on}\ \Omega\times(0, T), 
     \\[2mm]
        &(\ep\theta_{\ep} + \ln \theta_{\ep})(0) 
        = \ep\theta_{0} + \ln \theta_{0},\ 
        (\varphi_{\ep})(0) = \varphi_0,\ (\varphi_{\ep})_{t}(0) = v_0         
                                                   \quad \mbox{a.e.\ on}\ \Omega. 
     \end{align*}
 \end{df}

\smallskip

\noindent The following theorem asserts existence of weak solutions to \ref{Pep}. 
\begin{thm}\label{maintheorem2}
Assume that {\rm (C1)-(C4)} hold. 
Then for all $\ep \in (0, 1]$ 
there exists a weak solution $(\theta_{\ep}, \varphi_{\ep})$ of {\rm \ref{Pep}}.
\end{thm}

\medskip

To show existence for \ref{Pep}, 
in reference to \cite{CC2018, K7},  
we employ the following time discretization scheme: 
find $(\theta_{n+1}, \varphi_{n+1})$ such that 
\begin{equation*}\tag*{(P)$_{n}$}\label{Pn}
     \begin{cases}
         \frac{u_{n+1} - u_{n}}{h} + \ell\frac{\varphi_{n+1}-\varphi_{n}}{h}
         - \eta\Delta\theta_{n+1} = f_{n+1}   
         & \mbox{in}\ \Omega, 
 \\[2mm]
         z_{n+1} + v_{n+1} + a(\cdot)\varphi_{n} - J\ast\varphi_{n} 
         + \beta(\varphi_{n+1}) + \pi(\varphi_{n+1}) 
         = \ell\theta_{n+1} 
         & \mbox{in}\ \Omega, 
 \\[1mm]
         z_{n+1} = \frac{v_{n+1}-v_{n}}{h},\ v_{n+1} = \frac{\varphi_{n+1}-\varphi_{n}}{h} 
         & \mbox{in}\ \Omega, 
 \\[1mm]
         \partial_{\nu}\theta_{n+1} = 0                                   
         & \mbox{on}\ \partial\Omega 
     \end{cases}
\end{equation*}
for $n=0, ... , N-1$, where $h=\frac{T}{N}$, $N \in \mathbb{N}$, 
\begin{align}\label{uj}
&u_{j} := \ep\theta_{j} + \ln \theta_{j}
\end{align}
for $j=0, 1, ..., N$, and   
$f_{k} := \frac{1}{h}\int_{(k-1)h}^{kh} f(s)\,ds$  
for $k = 1, ... , N$. 
Indeed, we can prove existence for \ref{Pn}.   
\begin{thm}\label{maintheorem3}
Assume that {\rm (C1)-(C4)} hold. 
Then for all $\ep \in (0, 1]$ there exists $h_{0\ep} \in (0, 1]$  
such that for all $h \in (0, h_{0\ep})$ 
there exists a unique solution of {\rm \ref{Pn}} satisfying 
\[
\theta_{n+1} \in W,\ \varphi_{n+1} \in L^{\infty}(\Omega) 
\quad \mbox{for}\ n = 0, ..., N-1. 
\]
\end{thm}
In order to derive existence for \ref{Pep} 
by passing to the limit in \ref{Pn} as $h \searrow 0$, 
we put  
\begin{align}
&\widehat{u}_{h}(t) := u_{n} + \frac{u_{n+1} - u_{n}}{h}(t-nh), 
\label{hat1} 
\\[2mm]  
&\widehat{\varphi}_{h}(t) := \varphi_{n} + \frac{\varphi_{n+1}-\varphi_{n}}{h}(t-nh), 
\label{hat2} 
\\[2mm]  
&\widehat{v}_{h}(t) := v_{n} + \frac{v_{n+1}-v_{n}}{h}(t-nh)  
\label{hat3} 
\end{align}
for $t \in [nh, (n+1)h]$, $n = 0, ..., N-1$, 
and 
\begin{align}
&\overline{u}_{h}(t) := u_{n+1},\  
\overline{\theta}_{h} (t) := \theta_{n+1},\ 
\overline{\varphi}_{h} (t) := \varphi_{n+1},\ 
\underline{\varphi}_{h} (t) := \varphi_{n},\ \label{line1}   
\\[2mm]
&\overline{v}_{h} (t) := v_{n+1},\ 
\overline{z}_{h} (t) := z_{n+1},\ 
\overline{f}_{h}(t) := f_{n+1}  
\label{line2}   
\end{align}
for \ $t \in (nh, (n+1)h]$, $n=0, ..., N-1$, 
and we rewrite \ref{Pn} as  
\begin{equation*}\tag*{(P)$_{h}$}\label{Ph}
     \begin{cases}
          (\widehat{u}_{h})_{t} + \ell( \widehat{\varphi}_{h})_{t} 
         - \eta\Delta\overline{\theta}_{h} = \overline{f}_h    
         & \mbox{in}\ \Omega\times(0, T), 
 \\[2mm]
         \overline{z}_{h} + \overline{v}_{h} 
         + a(\cdot)\underline{\varphi}_{h} - J\ast\underline{\varphi}_{h} 
         + \beta(\overline{\varphi}_{h}) + \pi(\overline{\varphi}_{h}) 
         = \ell\overline{\theta}_{h} 
         & \mbox{in}\ \Omega\times(0, T), 
\\[2mm]
         \overline{z}_{h} = (\widehat{v}_{h})_{t},\ \overline{v}_{h} = (\widehat{\varphi}_{h})_{t} 
         & \mbox{in}\ \Omega\times(0, T), 
 \\[2mm]
         \overline{u}_{h} = \mbox{\rm Ln$_{\ep}$}(\overline{\theta}_{h})   
         & \mbox{in}\ \Omega\times(0, T),  
 \\[2mm]
         \partial_{\nu}\overline{\theta}_{h} = 0                                   
         & \mbox{on}\ \partial\Omega\times(0, T),
 \\[2mm]
        \widehat{u}_{h}(0) = \ep\theta_{0} + \ln \theta_{0},\ 
        \widehat{\varphi}_{h}(0) = \varphi_0,\    
        \widehat{v}_{h}(0) = v_0                                     
         & \mbox{in}\ \Omega.  
     \end{cases}
 \end{equation*}
Here we can check directly the following identities by \eqref{hat1}-\eqref{line2}: 
\begin{align}
&\|\widehat{u}_{h}\|_{L^{\infty}(0, T; H)} 
= \max\{ \|u_{0}\|_{H}, \|\overline{u}_{h}\|_{L^{\infty}(0, T; H)} \}, 
\label{tool1} \\[1mm] 
&\|\widehat{\varphi}_{h}\|_{L^{\infty}(0, T; L^{\infty}(\Omega))} 
= \max\{ \|\varphi_{0}\|_{L^{\infty}(\Omega)}, 
                            \|\overline{\varphi}_{h}\|_{L^{\infty}(0, T; L^{\infty}(\Omega))} \},  
\label{tool2}  \\[1mm] 
&\|\widehat{v}_{h}\|_{L^{\infty}(0, T; L^{\infty}(\Omega))} 
= \max\{ \|v_{0}\|_{L^{\infty}(\Omega)}, 
                                     \|\overline{v}_{h}\|_{L^{\infty}(0, T; L^{\infty}(\Omega))} \},  
\label{tool3}  \\[1mm]  
&\|\overline{u}_{h} - \widehat{u}_{h}\|_{L^2(0, T; V^*)}^2 
= \frac{h^2}{3}\|(\widehat{u}_{h})_{t}\|_{L^2(0, T; V^*)}^2, 
\label{tool4}  \\[1mm] 
&\|\overline{\varphi}_{h} - \widehat{\varphi}_{h}\|_{L^{\infty}(0, T; L^{\infty}(\Omega)} 
= h\|(\widehat{\varphi}_{h})_{t}\|_{L^{\infty}(0, T; L^{\infty}(\Omega))} 
= h\|\overline{v}_{h}\|_{L^{\infty}(0, T; L^{\infty}(\Omega))}, 
\label{tool5} \\[1mm]
&\|\overline{v}_{h} - \widehat{v}_{h}\|_{L^2(0, T; H)}^2 
= \frac{h^2}{3}\|(\widehat{v}_{h})_{t}\|_{L^2(0, T; H)}^2 
= \frac{h^2}{3}\|\overline{z}_{h}\|_{L^2(0, T; H)}^2,     
\label{tool6}  \\[1mm] 
&\underline{\varphi}_{h} = \overline{\varphi}_{h} - h(\widehat{\varphi}_{h})_{t}. 
\label{tool7}
\end{align}

\medskip

\begin{remark}\label{Keypoints}
In the case that $\zeta = 1$, 
$\alpha(\theta) = \theta$, $A\varphi = a(\cdot)\varphi - J\ast\varphi$, 
to establish the $L^2(0, T; H^2(\Omega))$-estimate for $\theta$ 
is a key to prove existence for \eqref{E}. 
On the other hand, in this paper,  
to derive the $L^{\infty}(0, T; H^2(\Omega))$-estimate for 
$\int_{0}^{t} \theta(s)\,ds$
is a key to show existence for \eqref{P}.    
More precisely, 
to obtain an estimate for $h\max_{1 \leq m \leq N} 
                          \left\|\sum_{n=0}^{m-1} \theta_{n+1} \right\|_{H^2(\Omega)}$ 
(see Lemma \ref{esth5}) 
is a key to prove existence for \eqref{P}. 
Also, to establish Cauchy's criteria for solutions of \ref{Ph} and \ref{Pep}, 
respectively, is a key to show existence for \eqref{P} 
(see Lemmas \ref{CauchyforPh} and \ref{CauchyforPep}).   
\end{remark}

\bigskip

This paper is organized as follows. 
Section \ref{Sec2} contains 
the proof of existence for the discrete problem \ref{Pn}. 
In Section \ref{Sec3} we deduce uniform estimates 
for \ref{Ph}. 
In Section \ref{Sec4} 
we derive Cauchy's criterion for solutions of \ref{Ph} 
and we prove existence of weak solutions to \ref{Pep} 
by passing to the limit in \ref{Ph} as $h \searrow 0$. 
Section \ref{Sec5} 
establishes uniform estimates and Cauchy's criterion for solutions of \ref{Pep} 
and show existence of weak solutions to \eqref{P} 
by passing to the limit in \ref{Pep} as $\ep \searrow 0$.   

\vspace{10pt}

\section{Existence for the discrete problem}\label{Sec2}
In this section we will show Theorem \ref{maintheorem3}.     
\begin{lem}\label{PT}
 Let $\gamma : D(\gamma) \subset \mathbb{R} \to \mathbb{R}$ 
 be a multi-valued maximal monotone function.     
 Then
    \begin{align*}
    \bigl(-\Delta u, \gamma_{\tau}(u)\bigr)_{H} \geq 0 
    \quad \mbox{for all}\ 
    u\in W\ \mbox{and all}\ \tau > 0,  
    \end{align*}
 where 
 $\gamma_{\tau}$ is the Yosida approximation of $\gamma$ on $\mathbb{R}$. 
 In particular, if $\gamma : D(\gamma) \subset \mathbb{R} \to \mathbb{R}$ 
 is a single-valued maximal monotone function, 
 then 
 \begin{align*}
    \bigl(-\Delta u, \gamma(u)\bigr)_{H} \geq 0 
    \quad \mbox{for all}\ 
    u\in W\ \mbox{with}\ \gamma(u)\in H. 
    \end{align*}
 \end{lem}
 \begin{proof}
 From Okazawa \cite[Proof of Theorem 3 with $a=b=0$]{O-1983}   
 we have that 
 \[
 \bigl(-\Delta u, \gamma_{\tau}(u)\bigr)_{H} \geq 0 
    \quad \mbox{for all}\ 
    u\in W\ \mbox{and all}\ \tau>0.
 \]
 In the case that $\gamma : D(\gamma) \subset \mathbb{R} \to \mathbb{R}$ 
 is a single-valued maximal monotone function,  
 since it holds that $\gamma_{\tau}(u) \to \gamma(u)$ in $H$  
 as $\tau \searrow 0$ if $\gamma(u) \in H$ 
 (see e.g., {\cite[Proposition 2.6]{Brezis}} 
 or \cite[Theorem IV.1.1]{S-1997}),   
 we can show the second inequality.       
 \end{proof}
\begin{lem}\label{elliptic1}
For all $g \in H$, $\ep \in (0, 1]$, $h > 0$ 
there exists a unique solution $\theta \in W$ of the equation 
\[
\ep\theta + \ln \theta- \eta h \Delta \theta = g 
\quad \mbox{a.e.\ on}\ \Omega.  
\] 
\end{lem}
\begin{proof}
Let $\tau > 0$ and 
let $\ln_{\tau}$ be the Yosida approximation of $\ln$ on $\mathbb{R}$. 
Moreover, we define the operator $\Phi : V \to V^*$ as 
\begin{equation*}
\langle \Phi\theta, w \rangle_{V^{*}, V} 
:= (\ep\theta + \ln_{\tau} \theta, w)_{H} 
+ \eta h\int_{\Omega} \nabla\theta \cdot \nabla w  
\quad \mbox{for}\ \theta, w \in V. 
\end{equation*}
Then we can confirm that this operator is monotone, continuous and coercive 
for all $\ep \in (0, 1]$ and all $\tau, h >0$.  
Indeed, it follows from the monotonicity and the Lipschitz continuity of $\ln_{\tau}$ that 
\begin{align*}
\langle\Phi\theta-\Phi\overline{\theta}, \theta-\overline{\theta}\rangle_{V^{*}, V} 
&=  \ep\|\theta-\overline{\theta}\|_{H}^2 
     + \int_{\Omega} 
          (\ln_{\tau} \theta - \ln_{\tau}\overline{\theta})(\theta-\overline{\theta})
   + \eta h\int_{\Omega} |\nabla(\theta-\overline{\theta})|^2 
\\
&\geq \min\{\ep, \eta h\}\|\theta-\overline{\theta}\|_{V}^2,     
\end{align*} 
\begin{align*}
|\langle \Phi\theta - \Phi\overline{\theta}, w \rangle_{V^{*}, V}| 
&= \left| 
    \ep\int_{\Omega} (\theta - \overline{\theta})w
    + \int_{\Omega} (\ln_{\tau}(\theta) - \ln_{\tau}(\overline{\theta}))w
   + \eta h\int_{\Omega} \nabla(\theta-\overline{\theta}) \cdot \nabla w   
     \right|
\\ 
&\leq \max\{\ep, \|\ln_{\tau}'\|_{L^{\infty}(\mathbb{R})}, \eta h\}
                                                      \|\theta-\overline{\theta}\|_{V}\|w\|_{V} 
\end{align*}
and 
\begin{align*}
\langle \Phi\theta - \ln_{\tau}(0), \theta \rangle_{V^{*}, V}  
&= \ep\|\theta\|_{H}^2 
     + \int_{\Omega} 
          (\ln_{\tau} \theta - \ln_{\tau}(0))(\theta-0)
   + \eta h\int_{\Omega} |\nabla\theta|^2  
\\ 
&\geq \min\{\ep, \eta h\}\|\theta\|_{V}^2            
\end{align*}
for all $\theta, \overline{\theta}, w \in V$, $\ep \in (0, 1]$, $\tau, h > 0$. 
Hence the operator $\Phi : V \to V^*$ is surjective 
for all $h \in \left(0, \frac{1}{\|\pi'\|_{L^{\infty}(\mathbb{R})}}\right)$ 
(see e.g., \cite[p.\ 37]{Barbu2})
and then we deduce from the elliptic regularity theory that 
for all $g \in H$, $\ep \in (0, 1]$, $\tau, h > 0$ 
there exists a unique solution $\theta_{\tau} \in W$ of the equation  
\begin{equation}\label{b1}
\ep\theta_{\tau} + \ln_{\tau}(\theta_{\tau}) -\eta h \Delta \theta_{\tau} = g 
\quad \mbox{a.e.\ on}\ \Omega.    
\end{equation}
Here, noting that $|\ln_{\tau}(1)| \leq |\ln 1| = 0$, 
we see from \eqref{b1} that 
\begin{equation}\label{b2}
\ep\theta_{\tau} + \ln_{\tau}(\theta_{\tau}) - \ln_{\tau}(1) 
- \eta h \Delta \theta_{\tau} = g \quad \mbox{a.e.\ on}\ \Omega.    
\end{equation} 
Multiplying \eqref{b2} by $\theta_{\tau} - 1$ and integrating over $\Omega$ 
imply that 
\begin{align*}
\ep(\theta_{\tau}, \theta_{\tau} - 1)_{H} 
+ (\ln_{\tau}(\theta_{\tau}) - \ln_{\tau}(1), \theta_{\tau} - 1)_{H} 
+ \eta h \|\nabla \theta_{\tau}\|_{H}^2 
= (g, \theta_{\tau} - 1)_{H}    
\end{align*}
and then we derive from the monotonicity of $\ln_{\tau}$ and the Young inequality 
that 
for all $\ep \in (0, 1]$ and all $h >0$ 
there exists a constant $C_{1} = C_{1}(\ep, h) > 0$ such that 
\begin{align}\label{b3}
\|\theta_{\tau}\|_{V} \leq C_{1} 
\end{align}
for all $\tau > 0$. 
We test \eqref{b1} by $\ln_{\tau}(\theta_{\tau})$, 
use the Young inequality, \eqref{b3} and Lemma \ref{PT} 
to infer that 
for all $\ep \in (0, 1]$ and all $h >0$ 
there exists a constant $C_{2} = C_{2}(\ep, h) > 0$ such that 
\begin{align}\label{b4}
\|\ln_{\tau}(\theta_{\tau})\|_{H} \leq C_{2}
\end{align}
for all $\tau > 0$. 
Combining \eqref{b1}, \eqref{b3} and \eqref{b4} means that 
for all $\ep \in (0, 1]$ and all $h >0$ 
there exists a constant $C_{3} = C_{3}(\ep, h) > 0$ such that 
\begin{align}\label{b5}
\|\theta_{\tau}\|_{W} \leq C_{3}
\end{align}
for all $\tau > 0$. 
Thus, owing to \eqref{b4}, \eqref{b5} 
and the compact embedding $W \hookrightarrow V$, 
there exist some functions $\theta \in W$, $\xi \in H$ such that 
\begin{align}
& \theta_{\tau} \to \theta \quad \mbox{weakly in}\ W, 
\label{b6} \\  
& \theta_{\tau} \to \theta \quad \mbox{strongly in}\ V, 
\label{b7} \\
& \ln_{\tau}(\theta_{\tau}) \to \xi \quad \mbox{weakly in}\ H \label{b8}
\end{align}
as $\tau = \tau_{j} \searrow 0$. 
We have from \eqref{b7} and \eqref{b8} that 
\begin{equation*}
(\ln_{\tau}(\theta_{\tau}), \theta_{\tau})_{H} 
\to (\xi, \theta)_{H} 
\end{equation*}
as $\tau = \tau_{j} \searrow 0$, 
which yields that 
\begin{equation}\label{b9}
\xi = \ln \theta \quad \mbox{a.e.\ on}\ \Omega  
\end{equation}
(see e.g., \cite[Lemma 1.3, p.\ 42]{Barbu1}).   
Therefore it follows from \eqref{b1}, \eqref{b6}, \eqref{b8}, and \eqref{b9} that 
\begin{equation}\label{b10}
\ep\theta + \ln \theta- \eta h \Delta \theta = g 
\quad \mbox{a.e.\ on}\ \Omega.  
\end{equation}
Moreover, we can prove uniqueness of solutions to \eqref{b10} 
by the monotonicity of $\ln$. 
\end{proof}
\begin{lem}\label{elliptic2}
For all $g \in H$ and all $h \in (0, \min\{1, 1/\|\pi'\|_{L^{\infty}(\mathbb{R})}\})$
there exists a unique solution $\varphi \in H$ of the equation 
\[
\varphi + h\varphi + h^{2} \beta(\varphi) + h^{2} \pi(\varphi) = g 
\quad \mbox{a.e.\ on}\ \Omega.  
\]
\end{lem}
\begin{proof}
We can obtain this lemma in reference to \cite[Lemma 2.1]{K7}. 
\end{proof}

\begin{prth1.3}
We can rewrite \ref{Pn} as 
\begin{equation*}\tag*{(Q)$_{n}$}\label{Qn}
     \begin{cases}
         \ep\theta_{n+1} + \ln \theta_{n+1} - \eta h\Delta\theta_{n+1} 
         = hf_{n+1} + \ell\varphi_{n} - \ell\varphi_{n+1} + \ep\theta_{n} + \ln \theta_{n},  
     \\[5mm] 
         \varphi_{n+1} + h\varphi_{n+1} + h^2 \beta(\varphi_{n+1}) + h^2 \pi(\varphi_{n+1}) 
         \\[2mm]
         = \ell h^2 \theta_{n+1} + \varphi_{n} + hv_{n} + h\varphi_{n}  
            -h^2 a(\cdot)\varphi_{n} + h^2 J\ast\varphi_{n}.  
     \end{cases}
 \end{equation*}
To show Theorem \ref{maintheorem3} 
it suffices to derive existence and uniqueness
of solutions to \ref{Qn} in the case that n = 0. 
Let $h \in (0, \min\{1, 1/\|\pi'\|_{L^{\infty}(\mathbb{R})}\})$. 
Then we see from Lemma \ref{elliptic1} that 
for all $\varphi \in H$ there exists a unique function $\overline{\theta} \in W$ such that  
\begin{align}\label{c1}
\ep\overline{\theta} + \ln \overline{\theta} - \eta h\Delta\overline{\theta}
         = hf_{1} + \ell\varphi_{0} - \ell\varphi +  \ep\theta_{0} + \ln \theta_{0}.  
\end{align}
Also, by Lemma \ref{elliptic2} it holds that 
for all $\theta \in H$ there exists a unique function $\overline{\varphi}$ 
such that  
\begin{align}\label{c2}
\overline{\varphi} + h\overline{\varphi} 
+ h^2 \beta(\overline{\varphi}) + h^2 \pi(\overline{\varphi}) 
= \ell h^2 \theta + \varphi_{0} + hv_{0} + h\varphi_{0}  
            -h^2 a(\cdot)\varphi_{0} + h^2 J\ast\varphi_{0}. 
\end{align}
Hence we can define ${\cal A} : H \to H$, ${\cal B} : H \to H$ 
and ${\cal S} : H \to H$ as   
\[
{\cal A}\varphi = \overline{\theta},\ {\cal B}\theta = \overline{\varphi} 
\quad \mbox{for}\ \varphi, \theta \in H
\]
and 
\[
{\cal S} = {\cal B} \circ {\cal A}, 
\]
respectively. 
Now we let $\varphi, \widetilde{\varphi} \in H$. 
Then we deduce from \eqref{c1} that 
\begin{align*}
&\ep\|{\cal A}\varphi - {\cal A}\widetilde{\varphi}\|_{H}^2 
+ (\ln ({\cal A}\varphi) - \ln ({\cal A}\widetilde{\varphi}), 
                                                      {\cal A}\varphi - {\cal A}\widetilde{\varphi})_{H} 
+ \eta h\|\nabla ({\cal A}\varphi - {\cal A}\widetilde{\varphi})\|_{H}^2 
\\ 
&= -\ell (\varphi - \widetilde{\varphi}, {\cal A}\varphi - {\cal A}\widetilde{\varphi})_{H} 
\\ 
&\leq \ell \|\varphi - \widetilde{\varphi}\|_{H}
                                                \|{\cal A}\varphi - {\cal A}\widetilde{\varphi}\|_{H},  
\end{align*}
and hence the monotonicity of $\ln$ leads to the inequality 
\begin{equation}\label{c3}
\|{\cal A}\varphi - {\cal A}\widetilde{\varphi}\|_{H} 
\leq \frac{\ell}{\ep}\|\varphi - \widetilde{\varphi}\|_{H}. 
\end{equation}
Also, we have from \eqref{c2} and (C3) that  
\begin{align*}
&(1+h)\|{\cal S}\varphi - {\cal S}\widetilde{\varphi}\|_{H}^2 
+ h^2 (\beta({\cal S}\varphi) - \beta({\cal S}\widetilde{\varphi}), 
                                        {\cal S}\varphi - {\cal S}\widetilde{\varphi})_{H} 
\\
&= \ell h^2 ({\cal A}\varphi - {\cal A}\widetilde{\varphi}, 
                              {\cal S}\varphi - {\cal S}\widetilde{\varphi})_{H} 
     - h^2 (\pi({\cal S}\varphi) - \pi({\cal S}\widetilde{\varphi}), 
                                        {\cal S}\varphi - {\cal S}\widetilde{\varphi})_{H} 
\\
&\leq \ell h^2 \|{\cal A}\varphi - {\cal A}\widetilde{\varphi}\|_{H}
                                 \|{\cal S}\varphi - {\cal S}\widetilde{\varphi}\|_{H} 
         + \|\pi'\|_{L^{\infty}(\mathbb{R})}h^2
                                  \|{\cal S}\varphi - {\cal S}\widetilde{\varphi}\|_{H}^2.  
\end{align*}
Thus it follows from the monotonicity of $\beta$ that   
\begin{equation}\label{c4}
\|{\cal S}\varphi - {\cal S}\widetilde{\varphi}\|_{H} 
\leq \frac{\ell h^2}{1 + h - \|\pi'\|_{L^{\infty}(\mathbb{R})}h^2}
                                                 \|{\cal A}\varphi - {\cal A}\widetilde{\varphi}\|_{H}. 
\end{equation}
Therefore we combine \eqref{c3} and \eqref{c4} to obtain that   
\begin{equation*}
\|{\cal S}\varphi - {\cal S}\widetilde{\varphi}\|_{H} 
\leq \frac{\ell^2 h^2}{\ep(1 + h - \|\pi'\|_{L^{\infty}(\mathbb{R})}h^2)}
                                                                       \|\varphi - \widetilde{\varphi}\|_{H}.   
\end{equation*}
Then for all $\ep \in (0, 1]$ there exists 
$h_{00\ep} \in (0, \min\{1, 1/\|\pi'\|_{L^{\infty}(\mathbb{R})}\})$ 
such that 
\[
\frac{\ell^2 h^2}{\ep(1 + h - \|\pi'\|_{L^{\infty}(\mathbb{R})}h^2)} \in (0, 1). 
\]
Hence ${\cal S} : H \to H$ is a contraction mapping in $H$ 
for all $\ep \in (0, 1]$ and all $h \in (0, h_{00\ep})$,  
and then the Banach fixed-point theorem implies that 
for all $\ep \in (0, 1]$ and all $h \in (0, h_{00\ep})$ 
there exists a unique function $\varphi_{1} \in H$ 
such that $\varphi_{1} = {\cal S}\varphi_{1} \in H$. 
Thus, for all $\ep \in (0, 1]$ and all $h \in (0, h_{00\ep})$, 
putting $\theta_{1} := {\cal A}\varphi_{1} \in W$, 
we see that there exists a unique pair $(\theta_{1}, \varphi_{1}) \in H^2$ 
satisfying \ref{Qn} in the case that $n = 0$.   
Now we confirm that $\varphi_{1} \in L^{\infty}(\Omega)$. 
Let $\ep \in (0, 1]$ and let $h \in (0, h_{00\ep})$. 
Then, since 
$g_{1} := \ell h^2 \theta_{1} + \varphi_{0} + hv_{0} + h\varphi_{0}  
            -h^2 a(\cdot)\varphi_{0} + h^2 J\ast\varphi_{0} \in L^{\infty}(\Omega)$ 
by $\theta_{1} \in W$, $W \subset L^{\infty}(\Omega)$ and (C1),  
we test the second equation in (Q)$_{0}$ by $\varphi_{1}(x)$ 
and use the Young inequality, (C3) to infer that  
\begin{align*}
&|\varphi_{1}(x)|^2 + h|\varphi_{1}(x)|^2 + h^2 \beta(\varphi_{1}(x))\varphi_{1}(x) 
\\ \notag 
&= g_{1}(x)\varphi_{1}(x) - h^2 (\pi(\varphi_{1}(x))-\pi(0))\varphi_{1}(x) 
                                                                         - h^2 \pi(0)\varphi_{1}(x) 
\\ \notag 
&\leq \frac{1}{2}\|g_{1}\|_{L^{\infty}(\Omega)}^2 
         + \frac{1}{2}|\varphi_{1}(x)|^2 
         + h^2 \|\pi'\|_{L^{\infty}(\mathbb{R})}|\varphi_{1}(x)|^2 
         + \frac{1}{2} h^2 |\varphi_{1}(x)|^2 
         + \frac{1}{2} h^2 |\pi(0)|^2.  
\end{align*}
Therefore, owing to the monotonicity of $\beta$, 
for all $\ep \in (0, 1]$ there exists $h_{0\ep} \in (0, h_{00\ep})$ such that 
for all $h \in (0, h_{0\ep})$ there exists a constant $C_{1} = C_{1}(\ep, h) > 0$ 
such that 
$|\varphi_{1}(x)| \leq C_{1}$ for a.a.\ $x \in \Omega$. 
\qed
\end{prth1.3}

\vspace{10pt}
 
\section{Uniform estimates for the discrete problem}\label{Sec3}

In this section we will establish a priori estimates for \ref{Ph}.  
\begin{lem}\label{esth1}
Let $h_{0\ep}$ be as in Theorem \ref{maintheorem3}. 
Then there exists a constant $C>0$ depending on the data 
such that 
for all $\ep \in (0, 1]$ there exists $h_{1\ep} \in (0, h_{0\ep})$ 
such that 
\begin{align*}
&\ep\|\overline{\theta}_{h}\|_{L^{\infty}(0, T; H)}^2 
+ \|\overline{\theta}_{h}\|_{L^{\infty}(0, T; L^1(\Omega))} 
+ \|\nabla\overline{\theta}_{h}\|_{L^2(0, T; H)}^2 
+ \|\overline{\varphi}_{h}\|_{L^{\infty}(0, T; H)}^2 
+ \|\overline{v}_{h}\|_{L^{\infty}(0, T; H)}^2  
\notag \\
&\leq C
\end{align*}
for all $h \in (0, h_{1\ep})$.  
\end{lem}
\begin{proof}
We can prove this lemma in reference to \cite{CC2018, K7}.  
We multiply the first equation in \ref{Pn} by $h\theta_{n+1}$ 
to derive that 
\begin{align}\label{d1}
&\frac{\ep}{2}\|\theta_{n+1}\|_{H}^2 - \frac{\ep}{2}\|\theta_{n}\|_{H}^2 
+ \frac{\ep}{2}\|\theta_{n+1}-\theta_{n}\|_{H}^2 
+ (\ln (\theta_{n+1})-\ln (\theta_{n}), \theta_{n+1})_{H}
\notag \\ 
&+ \ell h(v_{n+1}, \theta_{n+1})_{H} + \eta h\|\nabla \theta_{n+1}\|_{H}^2  
= h\int_{\Omega}f_{n+1} \theta_{n+1}.  
\end{align}
Here the inequality $e^{x} (x-y) \geq e^{x} - e^{y}$ ($x, y \in \mathbb{R}$) 
means that 
\begin{align}\label{d2}
&(\ln (\theta_{n+1})-\ln (\theta_{n}), \theta_{n+1})_{H} 
\notag \\
&= (e^{\ln (\theta_{n+1})}, \ln (\theta_{n+1})-\ln (\theta_{n}))_{H} 
\notag \\ 
&\geq \int_{\Omega} e^{\ln (\theta_{n+1})} - \int_{\Omega} e^{\ln (\theta_{n})}
= \int_{\Omega} \theta_{n+1} - \int_{\Omega} \theta_{n}. 
\end{align}
By the identity $v_{n+1} = \frac{\varphi_{n+1} - \varphi_{n}}{h}$ it holds that  
\begin{align}\label{d3}
\frac{1}{2}\|\varphi_{n+1}\|_{H}^2 - \frac{1}{2}\|\varphi_{n}\|_{H}^2 
+ \frac{1}{2}\|\varphi_{n+1} - \varphi_{n}\|_{H}^2 
= h (v_{n+1}, \varphi_{n+1})_{H}.   
\end{align}
It follows from testing the second equation in \ref{Pn} by $hv_{n+1}$ that 
\begin{align}\label{d4}
&\frac{1}{2}\|v_{n+1}\|_{H}^2 - \frac{1}{2}\|v_{n}\|_{H}^2 
+ \frac{1}{2}\|v_{n+1} - v_{n}\|_{H}^2 
+ h\|v_{n+1}\|_{H}^2 
+ (\beta(\varphi_{n+1}), \varphi_{n+1} - \varphi_{n})_{H} 
\notag \\ 
&= \ell h(\theta_{n+1}, v_{n+1})_{H} 
     - h(a(\cdot)\varphi_{n} - J\ast \varphi_{n}, v_{n+1})_{H} 
     - h(\pi(\varphi_{n+1}), v_{n+1})_{H}. 
\end{align}
Here we have from (C2) and the definition of the subdifferential that 
\begin{equation}\label{d5}
(\beta(\varphi_{n+1}), \varphi_{n+1} - \varphi_{n})_{H} 
\geq \int_{\Omega} \widehat{\beta}(\varphi_{n+1}) 
        - \int_{\Omega} \widehat{\beta}(\varphi_{n}). 
\end{equation}
Thus we combine \eqref{d1}-\eqref{d5}, 
sum up $n = 0, ..., m-1$ with $1 \leq m \leq N$, 
use (C1), (C3) and the Young inequality to infer that 
there exists a constant $C_{1} > 0$ 
such that  
\begin{align*}
&\frac{\ep}{2}\|\theta_{m}\|_{H}^2 
+ \int_{\Omega} \theta_{m} 
+ \eta h\sum_{n=0}^{m-1}\|\nabla\theta_{n+1}\|_{H}^2 
+ \frac{1}{2}\|\varphi_{m}\|_{H}^2 
+ \frac{1}{2}\|v_{m}\|_{H}^2  
+ \int_{\Omega} \widehat{\beta}(\varphi_{m})
\notag \\ 
&\leq \frac{\ep}{2}\|\theta_{0}\|_{H}^2 
+ \int_{\Omega} \theta_{0} 
+ \frac{1}{2}\|\varphi_{0}\|_{H}^2 
+ \frac{1}{2}\|v_{0}\|_{H}^2  
+ \int_{\Omega} \widehat{\beta}(\varphi_{0}) 
\notag \\ 
&\,\quad + h\sum_{n=0}^{m-1}\int_{\Omega} f_{n+1} \theta_{n+1}  
              - h\sum_{n=0}^{m-1} 
                    (a(\cdot)\varphi_{n} - J\ast \varphi_{n} + \pi(\varphi_{n+1}), v_{n+1})_{H}              
\notag \\ 
&\leq C_{1} 
         + \|\theta_{m}\|_{L^1(\Omega)}
                                        \int_{(m-1)h}^{mh} \|f(s)\|_{L^{\infty}(\Omega)}\,ds 
         + h\sum_{j=0}^{m-1}\|f_{j}\|_{L^{\infty}(\Omega)}\|\theta_{j}\|_{L^1(\Omega)}  
\notag \\ 
&\,\quad+ C_{1}h\sum_{n=0}^{m-1}\|\varphi_{n+1}\|_{H}^2 
         + C_{1}h\sum_{n=0}^{m-1}\|v_{n+1}\|_{H}^2  
\end{align*}
for all $\ep \in (0, 1]$, $h \in (0, h_{0\ep})$ and $m = 1, ..., N$. 
Moreover, 
since by (C4) there exists $h_{1} \in (0, 1)$ such that 
\[
\int_{(m-1)h}^{mh} \|f(s)\|_{L^{\infty}(\Omega)}\,ds \leq \frac{1}{2}
\]
for all $h \in (0, h_{1})$ and $m = 1, ..., N$, 
for all $\ep \in (0, 1]$ there exists $h_{01\ep} \in (0, \min\{h_{1}, h_{0\ep}\})$ such that 
\begin{align*}
&\frac{\ep}{2}\|\theta_{m}\|_{H}^2 
+ \frac{1}{2}\int_{\Omega} \theta_{m}  
+ \eta h\sum_{n=0}^{m-1}\|\nabla\theta_{n+1}\|_{H}^2 
\notag \\ 
&+ \left(\frac{1}{2} - C_{1}h \right)\|\varphi_{m}\|_{H}^2 
+ \left(\frac{1}{2} - C_{1}h \right)\|v_{m}\|_{H}^2  
+ \int_{\Omega} \widehat{\beta}(\varphi_{m})
\notag \\ 
&\leq C_{1} 
         + h\sum_{j=0}^{m-1}\|f_{j}\|_{L^{\infty}(\Omega)}\|\theta_{j}\|_{L^1(\Omega)}   
         + C_{1}h\sum_{j=0}^{m-1}\|\varphi_{j}\|_{H}^2 
         + C_{1}h\sum_{j=0}^{m-1}\|v_{j}\|_{H}^2  
\end{align*}
for all $h \in (0, h_{01\ep})$ and $m = 1, ..., N$. 
Therefore there exists a constant $C_{2} > 0$ such that 
for all $\ep \in (0, 1]$ there exists $h_{1\ep} \in (0, h_{01\ep})$ such that 
\begin{align*}
&\ep\|\theta_{m}\|_{H}^2 
+ \int_{\Omega} \theta_{m}  
+ h\sum_{n=0}^{m-1}\|\nabla\theta_{n+1}\|_{H}^2 
+ \|\varphi_{m}\|_{H}^2 
+ \|v_{m}\|_{H}^2  
+ \int_{\Omega} \widehat{\beta}(\varphi_{m})
\notag \\ 
&\leq C_{2} 
         + C_{2}h\sum_{j=0}^{m-1} \|f_{j}\|_{L^{\infty}(\Omega)}\|\theta_{j}\|_{L^1(\Omega)}   
         + C_{2}h\sum_{j=0}^{m-1}\|\varphi_{j}\|_{H}^2 
         + C_{2}h\sum_{j=0}^{m-1}\|v_{j}\|_{H}^2  
\end{align*}
for all $h \in (0, h_{1\ep})$ and $m = 1, ..., N$, which leads to Lemma \ref{esth1} 
by the discrete Gronwall lemma (see e.g., \cite[Prop.\ 2.2.1]{Jerome}). 
\end{proof}

\begin{lem}\label{esth2}
Let $h_{1\ep}$ be as in Lemma \ref{esth1}. 
Then there exists a constant $C>0$ depending on the data 
such that 
\begin{align*}
\|(\widehat{u}_{h})_{t}\|_{L^{2}(0, T; V^{*})}^2 
\leq C
\end{align*}
for all $\ep \in (0, 1]$ and all $h \in (0, h_{1\ep})$.  
\end{lem}
\begin{proof}
We can prove this lemma by the first equation in \ref{Ph} 
and Lemma \ref{esth1}.  
\end{proof}

\begin{lem}\label{esth3}
Let $h_{1\ep}$ be as in Lemma \ref{esth1}. 
Then there exists a constant $C>0$ depending on the data 
such that 
\begin{align*}
\|\overline{\theta}_{h}\|_{L^{2}(0, T; V)}^2 
\leq C
\end{align*}
for all $\ep \in (0, 1]$ and all $h \in (0, h_{1\ep})$.  
\end{lem}
\begin{proof}
We can show this lemma by Lemma \ref{esth1} 
and the Poincar\'e--Wirtinger inequality. 
\end{proof}

\begin{lem}\label{esth4}
Let $h_{1\ep}$ be as in Lemma \ref{esth1}. 
Then there exists a constant $C>0$ depending on the data 
such that 
for all $\ep \in (0, 1]$ there exists $h_{2\ep} \in (0, h_{1\ep})$ 
such that 
\begin{align*}
\|\ln \overline{\theta}_{h}\|_{L^{\infty}(0, T; H)}^2 
\leq C
\end{align*}
for all $h \in (0, h_{2\ep})$. 
\end{lem}
\begin{proof}
Multiplying the first equation in \ref{Pn} by $hu_{n+1}$ implies that 
\begin{align}\label{g1}
&\frac{1}{2}\|u_{n+1}\|_{H}^2 - \frac{1}{2}\|u_{n}\|_{H}^2 
+ \frac{1}{2}\|u_{n+1} - u_{n}\|_{H}^2 
+ \eta h (-\Delta \theta_{n+1}, \ep\theta_{n+1} + \ln \theta_{n+1})_{H} 
\notag \\ 
&= h(f_{n+1}, u_{n+1})_{H} - \ell h(v_{n+1}, u_{n+1})_{H}.   
\end{align}
Summing \eqref{g1} up $n = 0, ..., m-1$ with $1 \leq m \leq N$, 
using the Young inequality, Lemma \ref{esth1} and the discrete Gronwall lemma, 
we deduce from Lemma \ref{PT} that 
there exists a constant $C_{1} > 0$ such that 
for all $\ep \in (0, 1]$ there exists $h_{2\ep} \in (0, h_{1\ep})$ 
such that 
\[
\|\overline{u}_{h}\|_{L^{\infty}(0, T; H)}^2 \leq C_{1}  
\]
for all $h \in (0, h_{2\ep})$, which yields Lemma \ref{esth4} by Lemma \ref{esth1}. 
\end{proof}

\begin{lem}\label{esth5}
Let $h_{2\ep}$ be as in Lemma \ref{esth4}. 
Then there exists a constant $C>0$ depending on the data 
such that 
\begin{align*}
h \max_{1 \leq m \leq N} \left\| \sum_{n=0}^{m-1} \theta_{n+1} \right\|_{W}
\leq C
\end{align*}
for all $\ep \in (0, 1]$ and all $h \in (0, h_{2\ep})$.  
\end{lem}
\begin{proof}
From the first equation in \ref{Pn} we have 
\begin{align}\label{h1}
u_{n+1} - u_{n} + \ell v_{n+1} - \ell v_{n} - h\Delta \theta_{n+1} = hf_{n+1}. 
\end{align}
We sum \eqref{h1} up $n = 0, ..., m-1$ with $1 \leq m \leq N$ 
to see that 
\begin{align}
u_{m} + \ell v_{m}  - h\Delta \left(\sum_{n=0}^{m-1} \theta_{n+1} \right) 
= u_{0} + \ell v_{0} + h\sum_{n=0}^{m-1} f_{n+1}
\end{align}
and then it follows from Lemmas \ref{esth1} and \ref{esth4} that 
there exists a constant $C_{1} > 0$ such that 
\begin{equation}\label{h2}
h \max_{1 \leq m \leq N} \left\| 
                                     \Delta \left(\sum_{n=0}^{m-1} \theta_{n+1} \right) 
                                  \right\|_{H}
\leq C_{1}
\end{equation}
for all $\ep \in (0, 1]$ and all $h \in (0, h_{2\ep})$.  
On the other hand, owing to Lemma \ref{esth3}, 
there exists a constant $C_{2} > 0$ such that 
\begin{equation}\label{h3}
h \max_{1 \leq m \leq N} \left\|\sum_{n=0}^{m-1} \theta_{n+1} \right\|_{H}
\leq C_{2}
\end{equation}
for all $\ep \in (0, 1]$ and all $h \in (0, h_{2\ep})$. 
Therefore combining \eqref{h2} and \eqref{h3} means that 
there exists a constant $C_{3} > 0$ such that 
\begin{equation*}
h \max_{1 \leq m \leq N} \left\|\sum_{n=0}^{m-1} \theta_{n+1} \right\|_{W}
\leq C_{3}
\end{equation*}
for all $\ep \in (0, 1]$ and all $h \in (0, h_{2\ep})$. 
\end{proof}

\begin{lem}\label{esth6}
Let $h_{2\ep}$ be as in Lemma \ref{esth4}. 
Then there exists a constant $C>0$ depending on the data 
such that 
for all $\ep \in (0, 1]$ 
there exists $h_{3\ep} \in (0, h_{2\ep})$ 
such that 
\begin{align*}
\|\overline{v}_{h}\|_{L^{\infty}(\Omega\times(0, T))}^2 
+ \|\overline{\varphi}_{h}\|_{L^{\infty}(\Omega\times(0, T))}^2   
\leq C
\end{align*}
for all $h \in (0, h_{3\ep})$.
\end{lem}
\begin{proof}
By the Young inequality and the identity $v_{n+1} = \frac{\varphi_{n+1}-\varphi_{n}}{h}$ 
it holds that 
\begin{align}\label{i1}
&\frac{1}{2}|\varphi_{n+1}(x)|^2 - \frac{1}{2}|\varphi_{n}(x)|^2 
+ \frac{1}{2}|\varphi_{n+1}(x) - \varphi_{n}(x)|^2 
\\ \notag 
&= \varphi_{n+1}(x)(\varphi_{n+1}(x) - \varphi_{n}(x))
\\ \notag 
&= h\varphi_{n+1}(x)v_{n+1}(x) 
\\ \notag 
&\leq \frac{1}{2}h\|\varphi_{n+1}\|_{L^{\infty}(\Omega)}^2 
         + \frac{1}{2}h\|v_{n+1}\|_{L^{\infty}(\Omega)}^2.    
\end{align}
Testing the second equation in \ref{Pn} by $hv_{n+1}$ 
and using (C1) yield that 
there exists a constant $C_{1} > 0$ such that 
\begin{align}\label{i2}
&\frac{1}{2}|v_{n+1}(x)|^2 - \frac{1}{2}|v_{n}(x)|^2 
+ \frac{1}{2}|v_{n+1}(x) - v_{n}(x)|^2 
+ \beta(\varphi_{n+1}(x))(\varphi_{n+1}(x) - \varphi_{n}(x)) 
\notag \\ 
& = h\bigl(\ell\theta_{n+1}(x) -a(x)\varphi_{n}(x) + (J\ast\varphi_{n})(x) 
                  + \pi(0) - \pi(\varphi_{n+1}(x)) - \pi(0) \bigr)v_{n+1}(x) 
\notag \\
&\leq \ell h\theta_{n+1}(x)v_{n+1}(x) 
         + C_{1}h\|\varphi_{n}\|_{L^{\infty}(\Omega)}^2 
\notag \\ 
   &\,\quad+ \frac{\|\pi'\|_{L^{\infty}(\mathbb{R})}^2 }{2}h
                    \|\varphi_{n+1}\|_{L^{\infty}(\Omega)}^2                                                                                                 
      + \frac{|\pi(0)|^2}{2}h  
      + \frac{3}{2}h\|v_{n+1}\|_{L^{\infty}(\Omega)}^2 
\end{align}
for all $\ep \in (0, 1]$, $h \in (0, h_{2\ep})$ 
and a.a.\ $x \in \Omega$. 
Here the condition (C2) and the definition of the subdifferential 
imply that 
\begin{align}\label{i3}
\beta(\varphi_{n+1}(x))(\varphi_{n+1}(x) - \varphi_{n}(x)) 
\geq \widehat{\beta}(\varphi_{n+1}(x)) - \widehat{\beta}(\varphi_{n}(x)).  
\end{align}
Thus we derive from \eqref{i1}-\eqref{i3} that 
\begin{align*}
&\frac{1}{2}|\varphi_{m}(x)|^2 + \frac{1}{2}|v_{m}(x)|^2 + \widehat{\beta}(\varphi_{m}(x)) 
\\ \notag 
&\leq \frac{1}{2}\|\varphi_{0}\|_{L^{\infty}(\Omega)}^2 
         + \frac{1}{2}\|v_{0}\|_{L^{\infty}(\Omega)}^2 
         + \|\widehat{\beta}(\varphi_{0})\|_{L^{\infty}(\Omega)} 
\\ \notag 
    &\,\quad+ \ell h\sum_{n=0}^{m-1}\theta_{n+1}(x)v_{n+1}(x)  
         + C_{1}h\sum_{n=0}^{m-1}\|\varphi_{n}\|_{L^{\infty}(\Omega)}^2 
\\ \notag 
   &\,\quad + \frac{\|\pi'\|_{L^{\infty}(\mathbb{R})}^2 + 1}{2}h
                                        \sum_{n=0}^{m-1} \|\varphi_{n+1}\|_{L^{\infty}(\Omega)}^2 
+ 2h\sum_{n=0}^{m-1}\|v_{n+1}\|_{L^{\infty}(\Omega)}^2 
+ \frac{|\pi(0)|^2}{2} T. 
\end{align*}
On the other hand, since $\theta_{j} > 0$ a.e.\ on $\Omega$ for $j = 0, 1, ..., N$, 
it follows from Lemma \ref{esth5} 
and the continuity of the embedding $W \hookrightarrow L^{\infty}(\Omega)$ 
that there exists a constant $C_{2} > 0$ such that 
\begin{align*}
\ell h\sum_{n=0}^{m-1}\theta_{n+1}(x)v_{n+1}(x) 
&\leq \ell h\left(\max_{1 \leq m \leq N}\|v_{m}\|_{L^{\infty}(\Omega)}\right) 
                                 \left\|\sum_{n=0}^{m-1}\theta_{n+1}\right\|_{L^{\infty}(\Omega)} 
\notag \\ 
&\leq C_{2}\max_{1 \leq m \leq N}\|v_{m}\|_{L^{\infty}(\Omega)} 
= C_{2}\|\overline{v}_{h}\|_{L^{\infty}(\Omega\times(0, T))} 
\end{align*}
for all $\ep \in (0, 1]$, $h \in (0, h_{2\ep})$ 
and for a.a.\ $x \in \Omega$, $m = 1, ..., N$. 
Hence there exists a constant $C_{3} > 0$ such that 
\begin{align*}
&\frac{1}{2}|\varphi_{m}(x)|^2 + \frac{1}{2}|v_{m}(x)|^2 + \widehat{\beta}(\varphi_{m}(x)) 
\\ \notag 
&\leq C_{3} + C_{2}\|\overline{v}_{h}\|_{L^{\infty}(\Omega\times(0, T))} 
         + C_{1}h\sum_{n=0}^{m-1}\|\varphi_{n}\|_{L^{\infty}(\Omega)}^2 
\\ \notag 
   &\,\quad + \frac{\|\pi'\|_{L^{\infty}(\mathbb{R})}^2 + 1}{2}h
                                        \sum_{n=0}^{m-1} \|\varphi_{n+1}\|_{L^{\infty}(\Omega)}^2 
+ 2h\sum_{n=0}^{m-1}\|v_{n+1}\|_{L^{\infty}(\Omega)}^2 
\end{align*}
for all $\ep \in (0, 1]$, $h \in (0, h_{2\ep})$ 
and for a.a.\ $x \in \Omega$, $m = 1, ..., N$, and then 
the inequality 
\begin{align*}
&\frac{1}{2}\|\varphi_{m}\|_{L^{\infty}(\Omega)}^2 
+ \frac{1}{2}\|v_{m}\|_{L^{\infty}(\Omega)}^2 
\\ \notag 
&\leq C_{3} + C_{2}\|\overline{v}_{h}\|_{L^{\infty}(\Omega\times(0, T))} 
         + C_{1}h\sum_{n=0}^{m-1}\|\varphi_{n}\|_{L^{\infty}(\Omega)}^2 
\\ \notag 
   &\,\quad + \frac{\|\pi'\|_{L^{\infty}(\mathbb{R})}^2 + 1}{2}h
                                        \sum_{n=0}^{m-1} \|\varphi_{n+1}\|_{L^{\infty}(\Omega)}^2 
+ 2h\sum_{n=0}^{m-1}\|v_{n+1}\|_{L^{\infty}(\Omega)}^2 
\end{align*}
holds. 
Thus we see that 
\begin{align*}
&\frac{1 - (\|\pi'\|_{L^{\infty}(\mathbb{R})}^2 + 1)h}{2}
                                                \|\varphi_{m}\|_{L^{\infty}(\Omega)}^2 
+ \frac{1-4h}{2}\|v_{m}\|_{L^{\infty}(\Omega)}^2 
\\ \notag 
&\leq C_{3} + C_{2}\|\overline{v}_{h}\|_{L^{\infty}(\Omega\times(0, T))} 
         + \frac{2C_{1} + \|\pi'\|_{L^{\infty}(\mathbb{R})}^2 + 1}{2}h
                                        \sum_{j=0}^{m-1} \|\varphi_{j}\|_{L^{\infty}(\Omega)}^2 
\\ \notag 
&\,\quad+ 2h\sum_{j=0}^{m-1}\|v_{j}\|_{L^{\infty}(\Omega)}^2,  
\end{align*}
whence there exists a constant $C_{4} > 0$ such that 
for all $\ep \in (0, 1]$ 
there exists $h_{3\ep} \in (0, h_{2\ep})$ such that 
\begin{align*}
&\|\varphi_{m}\|_{L^{\infty}(\Omega)}^2 + \|v_{m}\|_{L^{\infty}(\Omega)}^2 
\\ \notag 
&\leq C_{4} + C_{4}\|\overline{v}_{h}\|_{L^{\infty}(\Omega\times(0, T))} 
         + C_{4}h\sum_{j=0}^{m-1} \|\varphi_{j}\|_{L^{\infty}(\Omega)}^2 
+ C_{4}h\sum_{j=0}^{m-1}\|v_{j}\|_{L^{\infty}(\Omega)}^2 
\end{align*}
for all $h \in (0, h_{3\ep})$ and $m = 1, ..., N$. 
Thus it follows from the discrete Gronwall lemma that 
there exists a constant $C_{5} > 0$ such that 
\begin{equation*}
\|\varphi_{m}\|_{L^{\infty}(\Omega)}^2 + \|v_{m}\|_{L^{\infty}(\Omega)}^2 
\leq C_{5} + C_{5}\|\overline{v}_{h}\|_{L^{\infty}(\Omega\times(0, T))} 
\end{equation*}
for all $\ep \in (0, 1]$, $h \in (0, h_{3\ep})$ and $m = 1, ..., N$. 
Therefore we have that 
\begin{align*}
\|\overline{\varphi}_{h}\|_{L^{\infty}(\Omega\times(0, T))}^2  
+ \|\overline{v}_{h}\|_{L^{\infty}(\Omega\times(0, T))}^2 
&\leq C_{5} + C_{5}\|\overline{v}_{h}\|_{L^{\infty}(\Omega\times(0, T))} 
\notag \\ 
&\leq C_{5} + \frac{1}{2}\|\overline{v}_{h}\|_{L^{\infty}(\Omega\times(0, T))}^2 
        + \frac{C_{5}^2}{2},  
\end{align*}
which implies that Lemma \ref{esth6} holds. 
\end{proof}

\begin{lem}\label{esth7}
Let $h_{3\ep}$ be as in Lemma \ref{esth6}. 
Then there exists a constant $C>0$ depending on the data 
such that 
\begin{align*}
\|\underline{\varphi}_{h}\|_{L^{\infty}(\Omega\times(0, T))}^2  
\leq C
\end{align*}
for all $\ep \in (0, 1]$ and all $h \in (0, h_{3\ep})$. 
\end{lem}
\begin{proof}
We can obtain this lemma by (C4) and Lemma \ref{esth6}.  
\end{proof}

\begin{lem}\label{esth8}
Let $h_{3\ep}$ be as in Lemma \ref{esth6}. 
Then there exists a constant $C>0$ depending on the data 
such that 
\begin{align*}
\|\overline{z}_{h}\|_{L^{2}(0, T; H)} 
\leq C
\end{align*}
for all $\ep \in (0, 1]$ and all $h \in (0, h_{3\ep})$. 
\end{lem}
\begin{proof}
Combining the second equation in \ref{Ph}, 
Lemmas \ref{esth3}, \ref{esth6}, \ref{esth7}, 
the continuity of $\beta$ and the condition (C3)  
leads to Lemma \ref{esth8}.    
\end{proof}

\begin{lem}\label{esth9}
Let $h_{3\ep}$ be as in Lemma \ref{esth6}. 
Then there exists a constant $C>0$ depending on the data 
such that 
\begin{align*}
\|\widehat{u}_{h}\|_{H^1(0, T; V^*) \cap L^{\infty}(0, T; H)} 
+ \|\widehat{v}_{h}\|_{H^1(0, T; H) \cap L^{\infty}(\Omega\times(0, T))} 
+ \|\widehat{\varphi}_{h}\|_{W^{1, \infty}(0, T; L^{\infty}(\Omega))} 
\leq C  
\end{align*}
for all $\ep \in (0, 1]$ and all $h \in (0, h_{3\ep})$. 
\end{lem}
\begin{proof}
we can prove this lemma by \eqref{tool1}-\eqref{tool3}, 
Lemmas \ref{esth1}, \ref{esth2}, \ref{esth4}, \ref{esth6} and \ref{esth8}.  
\end{proof}

\vspace{10pt}

\section{Existence for \ref{Pep}}\label{Sec4}

In this section we will prove existence of weak solutions to \ref{Pep}. 
\begin{lem}\label{CauchyforPh}
Let $h_{3\ep}$ be as in Lemma \ref{esth6}. 
Then there exists a constant $C>0$ depending on the data 
such that 
\begin{align*}
&\|\widehat{\varphi}_{h} - \widehat{\varphi}_{\tau}\|_{C([0, T]; H)} 
+ \|\widehat{v}_{h} - \widehat{v}_{\tau}\|_{C([0, T]; H)} 
+ \|\overline{v}_{h} - \overline{v}_{\tau}\|_{L^2(0, T; H)}
\notag \\ 
&\leq C(h^{1/2} + \tau^{1/2}) 
       + C\|\widehat{v}_{h} - \widehat{v}_{\tau}\|_{L^2(0, T; V^{*})}^{1/2} 
\end{align*}
for all $\ep \in (0, 1]$ and all $h, \tau \in (0, h_{3\ep})$. 
\end{lem}
\begin{proof}
We see from the identity $\overline{v}_{h}(s) = (\widehat{\varphi}_{h})_{s}(s)$ that 
\begin{align}\label{Cauchy1}
\frac{1}{2}\frac{d}{ds}\|\widehat{\varphi}_{h}(s) - \widehat{\varphi}_{\tau}(s)\|_{H}^2 
&= \bigl((\widehat{\varphi}_{h})_{s}(s) - (\widehat{\varphi}_{\tau})_{s}(s), 
                                    \widehat{\varphi}_{h}(s) - \widehat{\varphi}_{\tau}(s)\bigr)_{H} 
\notag \\   
&= (\overline{v}_{h}(s) - \overline{v}_{\tau}(s), 
                                    \widehat{\varphi}_{h}(s) - \widehat{\varphi}_{\tau}(s))_{H}.  
\end{align}
It follows from the identity $\overline{z}_{h}(s) = (\widehat{v}_{h})_{s}(s)$ that 
\begin{align}\label{Cauchy2}
&\frac{1}{2}\frac{d}{ds}\|\widehat{v}_{h}(s) - \widehat{v}_{\tau}(s)\|_{H}^2 
\notag \\  
&= \bigl((\widehat{v}_{h})_{s}(s) - (\widehat{v}_{\tau})_{s}(s), 
                                    \widehat{v}_{h}(s) - \widehat{v}_{\tau}(s)\bigr)_{H} 
\notag \\ 
&= (\overline{z}_{h}(s) - \overline{z}_{\tau}(s),  
                                    \widehat{v}_{h}(s) - \widehat{v}_{\tau}(s))_{H} 
\notag \\  
&= (\overline{z}_{h}(s) - \overline{z}_{\tau}(s), \widehat{v}_{h}(s) -\overline{v}_{h}(s))_{H} 
    + (\overline{z}_{h}(s) - \overline{z}_{\tau}(s), 
                                                     \overline{v}_{\tau}(s) - \widehat{v}_{\tau}(s))_{H} 
\notag \\ 
    &\,\quad+ (\overline{z}_{h}(s) - \overline{z}_{\tau}(s),  
                                              \overline{v}_{h}(s) - \overline{v}_{\tau}(s))_{H}.  
\end{align}
Here the second equation in \ref{Ph} yields that 
\begin{align}\label{Cauchy3}
&(\overline{z}_{h}(s) - \overline{z}_{\tau}(s),  
                                              \overline{v}_{h}(s) - \overline{v}_{\tau}(s))_{H} 
\notag \\ 
&= - \|\overline{v}_{h}(s) - \overline{v}_{\tau}(s)\|_{H}^2 
\notag \\ 
&\,\quad- \bigl(a(\cdot)(\underline{\varphi}_{h}(s) - \underline{\varphi}_{\tau}(s)) 
              - J\ast(\underline{\varphi}_{h}(s) - \underline{\varphi}_{\tau}(s)), 
                                         \overline{v}_{h}(s) - \overline{v}_{\tau}(s) \bigr)_{H} 
\notag \\ 
&\,\quad - \bigl(\beta(\overline{\varphi}_{h}(s)) - \beta(\overline{\varphi}_{\tau}(s)), 
                                                   \overline{v}_{h}(s) - \overline{v}_{\tau}(s)  \bigr)_{H} 
\notag \\
&\,\quad - \bigl(\pi(\overline{\varphi}_{h}(s)) - \pi(\overline{\varphi}_{\tau}(s)), 
                                                   \overline{v}_{h}(s) - \overline{v}_{\tau}(s)  \bigr)_{H} 
\notag \\
&\,\quad+ (\overline{\theta}_{h}(s) - \overline{\theta}_{\tau}(s), 
                                                       \overline{v}_{h}(s) - \overline{v}_{\tau}(s))_{H}. 
\end{align}
We have from \eqref{tool7} that 
\begin{align}\label{Cauchy4}
&\|\underline{\varphi}_{h}(s) - \underline{\varphi}_{\tau}(s)\|_{H}^2 
\notag \\ 
&= \|\overline{\varphi}_{h}(s) - h(\widehat{\varphi}_{h})_{s}(s) 
                   - \overline{\varphi}_{\tau}(s) + \tau(\widehat{\varphi}_{\tau})_{s}(s) \|_{H}^2 
\notag \\ 
& \leq 3\|\overline{\varphi}_{h}(s) - \overline{\varphi}_{\tau}(s) \|_{H}^2  
          + 3h^2\|(\widehat{\varphi}_{h})_{s}(s) \|_{H}^2 
          + 3\tau^2\|(\widehat{\varphi}_{\tau})_{s}(s) \|_{H}^2. 
\end{align}
It holds that 
\begin{align}\label{Cauchy5}
&\|\overline{\varphi}_{h}(s) - \overline{\varphi}_{\tau}(s) \|_{H}^2  
\notag \\
&= \|\overline{\varphi}_{h}(s) - \widehat{\varphi}_{h}(s) 
                                +  \widehat{\varphi}_{h}(s) - \widehat{\varphi}_{\tau}(s)
                                +\widehat{\varphi}_{\tau}(s) - \overline{\varphi}_{\tau}(s) \|_{H}^2 
\notag \\ 
&\leq 3\|\overline{\varphi}_{h}(s) - \widehat{\varphi}_{h}(s)\|_{H}^2 
        + 3\|\widehat{\varphi}_{h}(s) - \widehat{\varphi}_{\tau}(s)\|_{H}^2  
        + 3\|\widehat{\varphi}_{\tau}(s) - \overline{\varphi}_{\tau}(s)\|_{H}^2 
\end{align}
and 
\begin{align}\label{Cauchy6}
&(\overline{\theta}_{h}(s) - \overline{\theta}_{\tau}(s), 
                                                       \overline{v}_{h}(s) - \overline{v}_{\tau}(s))_{H} 
\notag \\ 
&= (\overline{\theta}_{h}(s) - \overline{\theta}_{\tau}(s), 
                                        \overline{v}_{h}(s) - \widehat{v}_{h}(s))_{H} 
   + (\overline{\theta}_{h}(s) - \overline{\theta}_{\tau}(s), 
                                                       \widehat{v}_{\tau}(s) - \overline{v}_{\tau}(s))_{H} 
\notag \\
&\,\quad+ \langle \widehat{v}_{h}(s) - \widehat{v}_{\tau}(s),                                               
                           \overline{\theta}_{h}(s) - \overline{\theta}_{\tau}(s) \rangle_{V^{*}, V}. 
\end{align}
Thus we derive from \eqref{Cauchy1}-\eqref{Cauchy6}, 
the integration over $(0, t)$, where $t \in [0, T]$, 
the Schwarz inequality, the Young inequality, 
(C1), Lemma \ref{esth6}, 
the local Lipschitz continuity of $\beta$, 
(C3), \eqref{tool1}-\eqref{tool3}, \eqref{tool5}, \eqref{tool6}, 
Lemmas \ref{esth3} and \ref{esth8} that 
there exists a constant $C_{1} > 0$ such that 
\begin{align*}
&\|\widehat{\varphi}_{h}(t) - \widehat{\varphi}_{\tau}(t)\|_{H}^2 
+ \|\widehat{v}_{h}(t) - \widehat{v}_{\tau}(t)\|_{H}^2 
\notag \\
&\leq C_{1}\int_{0}^{t}\|\widehat{\varphi}_{h}(s) - \widehat{\varphi}_{\tau}(s)\|_{H}^2\,ds  
       + C_{1}\|\widehat{v}_{h} - \widehat{v}_{\tau}\|_{L^2(0, T; V^{*})} 
\end{align*}
for all $\ep \in (0, 1]$ and all $h, \tau \in (0, h_{3\ep})$. 
Therefore we can obtain Lemma \ref{CauchyforPh} by the Gronwall lemma. 
\end{proof}

\medskip

\begin{prth1.2}
We see from Lemmas \ref{esth1}, \ref{esth3}, \ref{esth4}, \ref{esth6}-\ref{CauchyforPh}, 
the Aubin--Lions lemma for the compact embedding $H \hookrightarrow V^*$,  
the properties \eqref{tool4}-\eqref{tool7} that 
there exist some functions $\theta_{\ep}$, $w_{\ep}$, $\varphi_{\ep}$ such that 
\begin{align*}
    &\theta_{\ep} \in L^2(0, T; V) \cap L^{\infty}(0, T; H),\   
      \ep\theta_{\ep} + w_{\ep} \in H^1(0, T; V^{*}),\  
      w_{\ep} \in L^{\infty}(0, T; H),      \\[1mm]       
    &\varphi_{\ep} \in W^{2, 2}(0, T; H) \cap W^{1, \infty}(0, T; L^{\infty}(\Omega)) 
    \end{align*}
and 
\begin{align}
&\widehat{u}_{h} \to \ep\theta_{\ep} + w_{\ep} 
\quad \mbox{weakly$^*$ in}\ H^1(0, T; V^*) \cap L^{\infty}(0, T; H), 
\label{weakh1} \\[1.5mm] 
&\widehat{u}_{h} \to \ep\theta_{\ep} + w_{\ep} 
\quad \mbox{strongly in}\ C([0, T]; V^*),  
\label{strongh1} \\[1.5mm] 
&\ln \overline{\theta}_{h} \to w_{\ep} 
\quad \mbox{weakly$^*$ in}\ L^{\infty}(0, T; H), 
\label{weakh2} \\[1.5mm] 
&\overline{\theta}_{h} \to \theta_{\ep} 
\quad \mbox{weakly in}\ L^2(0, T; V), 
\label{weakh3} \\[1.5mm] 
&\overline{z}_{h} \to (\varphi_{\ep})_{tt} 
\quad \mbox{weakly in}\ L^2(0, T; H), 
\label{weakh4} \\[1.5mm] 
&\widehat{v}_{h} \to (\varphi_{\ep})_{t} 
\quad \mbox{strongly in}\ L^2(0, T; V^*),  
\notag \\[1.5mm]
&\widehat{v}_{h} \to (\varphi_{\ep})_{t} 
\quad \mbox{strongly in}\ C([0, T]; H), 
\label{strongh2} \\[1.5mm] 
&\overline{v}_{h} \to (\varphi_{\ep})_{t} 
\quad \mbox{weakly$^*$ in}\  L^{\infty}(\Omega\times(0, T)),   
\label{weakh5} \\[1.5mm] 
&\widehat{\varphi}_{h} \to \varphi_{\ep} 
\quad \mbox{weakly$^*$ in}\ W^{1, \infty}(0, T; L^{\infty}(\Omega)), 
\label{weakh6} \\[1.5mm]  
&\widehat{\varphi}_{h} \to \varphi_{\ep} 
\quad \mbox{strongly in}\ C([0, T]; H), 
\label{strongh3} \\[1.5mm]  
&\overline{\varphi}_{h} \to \varphi_{\ep} 
\quad \mbox{weakly$^*$ in}\ L^{\infty}(\Omega\times(0, T)),   
\label{weakh7} \\[1.5mm] 
&\underline{\varphi}_{h} \to \varphi_{\ep} 
\quad \mbox{weakly$^*$ in}\ L^{\infty}(\Omega\times(0, T)),   
\label{weakh8}  
\end{align}
as $h = h_{j} \searrow 0$. 
It follows from \eqref{tool4}, Lemma \ref{esth9} and  \eqref{strongh1} that 
\begin{align}\label{q1}
&\|\overline{u}_{h} - (\ep\theta_{\ep} + w_{\ep})\|_{L^2(0, T; V^*)} 
\notag \\ 
&\leq \|\overline{u}_{h} - \widehat{u}_{h}\|_{L^2(0, T; V^*)} 
       + \|\widehat{u}_{h} - (\ep\theta_{\ep} + w_{\ep})\|_{L^2(0, T; V^*)}  
\notag \\ 
&\leq \frac{h}{\sqrt{3}}\|(\widehat{u}_{h})_{t}\|_{L^2(0, T; V^*)}  
         + T^{1/2}\|\widehat{u}_{h} - (\ep\theta_{\ep} + w_{\ep})\|_{C([0, T]; V^*)}  
\notag \\
&\to 0
\end{align}
as $h = h_{j} \searrow 0$. 
We have from 
the identity $\overline{u}_{h} = \ep\overline{\theta}_{h} + \ln \overline{\theta}_{h}$, 
\eqref{q1}, \eqref{weakh3} 
that 
\begin{align*}
&\limsup_{h_{j} \searrow 0} 
   \int_{0}^{T} (\ln \overline{\theta}_{h}(t), \overline{\theta}_{h}(t))_{H}\,dt 
\notag \\ 
&= \limsup_{h_{j} \searrow 0} 
      \left(
         \int_{0}^{T} \langle \overline{u}_{h}(t), \overline{\theta}_{h}(t) \rangle_{V^*, V}\,dt 
         - \ep\int_{0}^{T} \|\overline{\theta}_{h}(t)\|_{H}^2\,dt   
       \right) 
\notag \\ 
&\leq \limsup_{h_{j} \searrow 0} 
           \int_{0}^{T} \langle \overline{u}_{h}(t), \overline{\theta}_{h}(t) \rangle_{V^*, V}\,dt 
         - \ep\liminf_{h_{j} \searrow 0} \int_{0}^{T} \|\overline{\theta}_{h}(t)\|_{H}^2\,dt   
\notag \\ 
&\leq \int_{0}^{T} \langle \ep\theta_{\ep}(t) + w_{\ep}(t), 
                                                           \theta_{\ep}(t) \rangle_{V^*, V}\,dt 
         - \ep\int_{0}^{T} \|\theta_{\ep}(t)\|_{H}^2\,dt   
\notag \\ 
&= \int_{0}^{T} (w_{\ep}(t), \theta_{\ep}(t))_{H}\,dt,  
\end{align*}
whence it holds that 
\begin{equation}\label{q2}
w_{\ep} = \ln \theta_{\ep}  \quad \mbox{a.e.\ on}\ \Omega\times (0, T)
\end{equation}
(see, e.g., \cite[Lemma 1.3, p.\ 42]{Barbu1}). 
On the other hand, 
we derive from \eqref{tool5}, Lemma \ref{esth6} and \eqref{strongh3} that 
\begin{align}\label{q3}
\|\overline{\varphi}_{h} - \varphi\|_{L^{\infty}(0, T; H)} 
&\leq \|\overline{\varphi}_{h} - \widehat{\varphi}_{h}\|_{L^{\infty}(0, T; H)} 
       + \|\widehat{\varphi}_{h} - \varphi\|_{L^{\infty}(0, T; H)} 
\notag \\ 
&= h\|\overline{v}_{h}\|_{L^{\infty}(0, T; H)} 
         + \|\widehat{\varphi}_{h} - \varphi\|_{C([0, T]; H)} 
\notag \\ 
&\leq |\Omega|^{1/2} h \|\overline{v}_{h}\|_{L^{\infty}(\Omega\times(0, T))} 
         + \|\widehat{\varphi}_{h} - \varphi\|_{C([0, T]; H)} 
\notag \\ 
&\to 0
\end{align}
as $h = h_{j} \searrow 0$. 
Therefore, thanks to \eqref{weakh1}-\eqref{weakh8}, \eqref{q2}, \eqref{q3},  
(C1), Lemma \ref{esth6}, the local Lipschitz continuity of $\beta$, 
and (C3), we can prove existence of weak solutions to \ref{Pep} 
by passing to the limit in \ref{Ph} as $h = h_{j} \searrow 0$.  
\qed 
\end{prth1.2}

\section{Existence for \eqref{P}}\label{Sec5}

In this section we will prove existence of weak solutions to \eqref{P}. 
\begin{lem}\label{estforPep}
There exists a constant $C>0$ depending on the data such that 
\begin{align*}
&\ep^{1/2}\|\theta_{\ep}\|_{L^{\infty}(0, T; H)} 
+ \|\theta_{\ep}\|_{L^2(0, T; V)} 
+ \|(\ep\theta_{\ep} + \ln \theta_{\ep})_{t}\|_{L^2(0, T; V^*)} 
+ \|\ln \theta_{\ep}\|_{L^{\infty}(0, T; H)} 
\notag \\ 
&+ \|\varphi_{\ep}\|_{W^{2, 2}(0, T; H) \cap W^{1, \infty}(0, T; L^{\infty}(\Omega))} 
\leq C 
\end{align*}
for all $\ep \in (0, 1]$. 
\end{lem}
\begin{proof}
We can obtain this lemma 
by Lemmas \ref{esth1}-\ref{esth4}, \ref{esth6}, \ref{esth9}. 
\end{proof}

\begin{lem}\label{CauchyforPep}
There exists a constant $C>0$ depending on the data such that 
\begin{align*}
\|\varphi_{\ep} - \varphi_{\gamma}\|_{C([0, T]; H)} 
+ \|v_{\ep} - v_{\gamma}\|_{C([0, T]; H)} 
+ \|v_{\ep} - v_{\gamma}\|_{L^2(0, T; H)}
\leq C\|v_{\ep} - v_{\gamma}\|_{L^2(0, T; V^{*})}^{1/2} 
\end{align*}
for all $\ep, \gamma \in (0, 1]$, where $v_{\ep} = (\varphi_{\ep})_{t}$.   
\end{lem}
\begin{proof}
We can show this lemma 
in a similar way to the proof of Lemma \ref{CauchyforPh}. 
\end{proof}

\begin{prth1.1}
Owing to Lemmas \ref{estforPep} and \ref{CauchyforPep}, 
we can establish existence of weaks solutions to \eqref{P} 
in a similar way to the proof of Theorem \ref{maintheorem2}. 
\qed
\end{prth1.1}

 
%
%
%


\begin{thebibliography}{99} 
%
\bibitem{Barbu1}
V. Barbu, 
``Nonlinear Semigroups and Differential Equations in Banach spaces'',  
Noordhoff International Publishing, Leyden, 1976. 
%
\bibitem{Barbu2}
V. Barbu,
``Nonlinear Differential Equations of Monotone Types in Banach Spaces'',
Springer, New York, 2010.
%
\bibitem{Brezis}
    H. Br\'ezis, ``Op\'erateurs Maximaux Monotones et 
    Semi-groupes de Contractions 
    dans les Especes de Hilbert'', 
    North-Holland, Amsterdam, 1973. 
%
\bibitem{CC2012}
G.\ Canevari, P.\ Colli, 
{\it Solvability and asymptotic analysis of a generalization 
of the Caginalp phase field system}, 
Commun.\ Pure Appl.\ Anal.\ {\bf 11}  (2012), 1959--1982. 
%
\bibitem{C2018} 
M. Colturato, 
{\it Well-posedness and longtime behavior for a singular phase field system 
with perturbed phase dynamics}, 
Evol. Equ. Control Theory {\bf 7} (2018), 217--245.
%
\bibitem{CC2018}
P. Colli, M. Colturato, 
{\it Global existence for a singular phase field system related to a sliding mode control problem}, 
Nonlinear Anal. Real World Appl.\ {\bf 41} (2018), 128--151. 
%
\bibitem{CK1}
P. Colli, S. Kurima, 
    {\it Time discretization of a nonlinear phase field system 
in general domains}, Comm.\ Pure Appl.\ Anal.\ {\bf 18} (2019), 3161--3179.   
%
\bibitem{EllZheng}
C.M.\ Elliott, S.\ Zheng,  
{\it Global existence and stability of solutions to the phase-field equations, 
in ``Free Boundary Problems''}, 
Internat.\ Ser.\ Numer.\ Math.\ {\bf 95}, 46--58, 
Birkh\"auser
Verlag, Basel, (1990).
%
\bibitem{GPS}
M. Grasselli, H. Petzeltov\'a,  G. Schimperna, 
{\it A nonlocal phase-field system with inertial term},  
Quart.\ Appl. Math.\ {\bf 65} (2007), 451--469. 
%
\bibitem{Jerome}
J.W.\ Jerome, 
``Approximations of Nonlinear Evolution Systems'',
Mathematics in Science and Engineering {\bf 164},
Academic Press Inc., Orlando, 1983. 
%
\bibitem{K4}
S. Kurima, 
{\it Time discretization of an initial value problem for 
a simultaneous abstract evolution equation applying to  
parabolic-hyperbolic phase-field systems}, 
ESAIM Math. Model. Numer. Anal.\ {\bf 54} (2020), 977--1002. 
%
\bibitem{K7}
S. Kurima, 
    {\it Time discretization of a nonlocal phase-field system with inertial term}, 
submitted, arXiv:2102.00860 [math.NA]. 
%
\bibitem{MN2018}
A.\ Miranville,  A.J.\  Ntsokongo, 
{\it On anisotropic Caginalp phase-field type models with singular nonlinear terms}, 
J.\ Appl.\ Anal.\ Comput.\  {\bf 8}  (2018), 655--674.
%
\bibitem{M2015}
T.\ Miyasita, 
{\it Global existence and exponential attractor of solutions of Fix-Caginalp equation}, 
Sci.\ Math.\ Jpn.\ {\bf 77}  (2015), 339--355. 
%
\bibitem{O-1983}
     N. Okazawa, 
     {\it An application of the perturbation theorem for 
                                     m-accretive operators},
     Proc.\ Japan Acad.\ Ser.\ A Math.\ Sci.\ {\bf 59} (1983), 
     88--90. 
%
\bibitem{S-1997}
    R. E. Showalter, 
    ``Monotone Operators in Banach Space 
    and 
    Nonlinear Partial Differential Equations'',
    Mathematical Surveys and Monographs, {\bf 49}, 
    American Mathematical Society, Providence, RI, 1997.  
%
\bibitem{WGZ2007}
H. Wu, M. Grasselli, S. Zheng, 
{\it Convergence to equilibrium for a parabolic-hyperbolic phase-field system 
with Neumann boundary conditions},  
Math.\ Models Methods Appl.\ Sci.\ {\bf 17} (2007), 125--153. 
%
\bibitem{WGZ2007dynamicalBD}
H. Wu, M. Grasselli, S. Zheng, 
{\it Convergence to equilibrium for a parabolic-hyperbolic phase-field system 
with dynamical boundary condition},  
J.\ Math.\ Anal.\ Appl.\ {\bf 329} (2007), 948-976. 
%
\end{thebibliography}
\end{document}